\documentclass[10pt]{amsart}
\usepackage{amsthm,amsfonts,amsmath,amssymb,mathrsfs,enumitem, eucal}
\usepackage[usenames]{color}
\usepackage[colorlinks=true, linkcolor=webgreen, citecolor=webgreen, urlcolor=webbrown]{hyperref}
\definecolor{webgreen}{rgb}{0,.5,0}
\definecolor{webbrown}{rgb}{.5,0,0}

\newtheorem{Thm}{Theorem}[section]
\newtheorem{Def}[Thm]{Definition}
\newtheorem{Lm}[Thm]{Lemma}
\newtheorem{Prop}[Thm]{Proposition}
\newtheorem{Cor}[Thm]{Corollary}

\theoremstyle{definition}
\newtheorem*{ack}{Acknowledgements}

\theoremstyle{remark}

\newtheorem{Rem}[Thm]{Remark}
\newtheorem{Exp}[Thm]{Example}

\numberwithin{equation}{section}

\def\<{\langle}
\def\>{\rangle}

\DeclareMathOperator{\gr}{gr}
\DeclareMathOperator{\id}{id}
\DeclareMathOperator{\ad}{ad}

\def\Q{\mathcal{Q}}

\def\Se{\mathcal{S}}
\def\T{\mathcal{T}}

\def\e{\varepsilon}
\def\c{\mathfrak{c}}

\begin{document}
\title[]
{Relative PBW type theorems for symmetrically braided Hopf algebras}
\author[]{Bogdan Ion}
\thanks{Department of Mathematics, University of Pittsburgh, Pittsburgh, PA15260}%{ and University of Bucharest, Faculty of Mathematics and Computer Science, Algebra and Number Theory research center, 14 Academiei St., Bucharest, Romania}
\thanks{E-mail address: \nolinkurl{bion@pitt.edu}}
%\thanks{Work partially supported by  CNCSIS grant 24/28.09.07.}
\date{April 24, 2010}% 
%\subjclass[2000]{14L15, 16W30, 16W70}
%\address{Department of Mathematics, University of Pittsburgh, Pittsburgh, PA 15260}
%\address{University of Bucharest, Faculty of Mathematics and Computer Science, Algebra and Number Theory research center, 14 Academiei St., Bucharest, Romania}
%\email{}
%\begin{abstract}
%\end{abstract}
\maketitle
%%%%%%%%%%%%%%%%%%%%%%%%%%%%%%%%%%%%%%%%%%%%%%%%%%%%%%%%%%%%%%%%%%%%%%%%%%%%%

\section*{Introduction} 

The fundamental result that opens the way to the representation theory of enveloping algebras of Lie algebras is the classical Poincar\'e-Birkhoff-Witt theorem. To what extent similar facts are true for general Hopf algebras it is not yet understood. For noncommutative cosemisimple Hopf algebras it is perhaps unreasonable to expect such a result as group algebras fall within this class of Hopf algebras. A more realistic question is whether a Hopf algebra has a PBW basis (perhaps in a generalized sense) as a module over its coradical in case this is a subalgebra, or else as a module over a Hopf subalgebra containing the coradical. Keeping in mind that the degree filtration of a enveloping algebra of a Lie algebra is nothing else than the coradical filtration we consider the following context.

Let $H$ be a Hopf algebra and let $K$ be a Hopf subalgebra containing the coradical $H_0$ of $H$. The Hopf subalgebra $K$ induces a Hopf algebra filtration on $H$ given by $F_nH=\wedge^{n+1}K$; if $K=H_0$ this is the coradical filtration of $H$. We denote by $\gr_K(H)$ the graded Hopf algebra associated to this filtration. Denote by $K^+$ the augmentation ideal of $K$ and let $J:=H/HK^+$; this is a quotient left $H$-module coalgebra of $H$ which is pointed irreducible as a coalgebra. The filtration on $J$ induced by the one on $H$ is the precisely the coradical filtration of $J$; its associated graded coalgebra is denoted by $\gr J$. Then, by \cite[Proposition 5.1]{masuoka} there is a unit and counit preserving left $K$-linear right $J$-colinear isomorphism between $H$ and $K\otimes J$ which preserves the filtrations on both sides and, of course it induces an appropriate isomorphism between $\gr_K(H)$ and $K\otimes \gr J$. Fortunately, $\gr J$ is not only a coalgebra but it has in fact a braided graded Hopf algebra structure. It is therefore reasonable to say that $H$ is of PBW type as a $K$-module if $\gr J$ is isomorphic as a braided graded algebra to the graded symmetric algebra of a braided graded vector space. Implicitly, if $H$ is of PBW type as a $K$-module then $J$ is linearly isomorphic to a braided symmetric algebra. The issue of $H$ having a PBW type basis as a $K$-module depends entirely on whether braided symmetric algebras for the appropriate braiding admit a basis consisting of monomials. If $\gr J$ is isomorphic to the braided symmetric algebra $\Se(V)$ associated to the braided vector space $V$ and $\Se(V)$ has a basis consisting of monomials with increasing letters in a well-ordered linear basis of $V$ then we say that $H$ has a PBW type basis as a $K$-module.

From this discussion, it appears that the most relevant context for such considerations is that of braided Hopf algebras. Hopf algebras in braided tensor categories, more specifically the irreducible ones,  are also of interest, first due to their role in the Andruskiewitsch-Schneider classification scheme  of pointed Hopf algebras (see e.g. \cite{andrus}), and second due to the development of supersymmetry and related considerations in the context of symmetrically braided categories. When $H$ is a braided Hopf algebra (not necessarily categorical) and $K$ a categorical braided Hopf subalgebra containing the coradical $H_0$, the same constructions are legitimate, Masuoka's result holds in this generality, and $\gr J$ is again a braided Hopf algebra but for a different braiding; all the concepts can be appropriately generalized to this context. We recall these structural results in some detail in Section \ref{appl}. 
\begin{Def} Let $H\supseteq K$ be a braided Hopf algebra and, respectively, a categorical braided Hopf subalgebra containing the coradical of $H$. Denote by $J$ the braided quotient coalgebra $H/HK^+$, and by $\gr J$ its associated braided graded coalgebra with respect to the coradical filtration. We say that $H$ is of PBW type as a $K$-module if the associated braided graded algebra $\gr J$ is isomorphic as a braided graded algebra to a braided graded symmetric algebra.
\end{Def}
It turns out that whether such an $H$ acquires a PBW type basis as a $K$-module depends entirely on the nature of the induced braiding on vector space $\gr_+ J/(\gr_+ J)^2$, where $\gr_+ J$ denotes the maximal homogeneous ideal of graded algebra $\gr J$.

If $H$ is irreducible and $K=H_0=k$, the field of definition, then being of PBW type as a $K$-module means simply that $\gr H$ is isomorphic to a braided symmetric algebra. The classical PBW theorem confirms that enveloping algebras of Lie algebras are of PBW type as $k$-modules. By a result of Kharchenko \cite[\S 7]{khar}, in characteristic zero enveloping algebras of symmetrically braided Lie algebras (which are irreducible) are of PBW type as $k$-modules. We show that in characteristic zero all irreducible symmetrically braided Hopf algebras are of PBW type as $k$-modules. The result is essentially an outcome of the work of  Masuoka \cite{masuoka}. By Theorem \ref{almostcentral}, in characteristic zero, if $H$ is symmetrically braided then $H$ is of PBW type as a $K$-module if either $K\hookrightarrow \gr_K(H)$ is central, or $\gr_K(H)\twoheadrightarrow K$ is cocentral. In particular, in characteristic zero, if the braiding on $H$ is trivial (i.e. $H$ is a usual Hopf algebra) and either then $K\hookrightarrow \gr_K(H)$ is central, or $\gr_K(H)\twoheadrightarrow K$ is cocentral, then $H$ has a PBW basis as a $K$-module. For PBW type theorems for a class of pointed Hopf algebras, which from the point of view outlined above reduce to statements about irreducible braided Hopf algebras with diagonal braiding, see \cite{khar1}.
%%%%%%%%%%%%%%%%%%%%%%%%%%%%%%%%%%%%%%%%%%%%%%%%%%%%%%%%%%%%%%%%%%%%%%%%%%%%%
\begin{ack} I am grateful to the referee for suggestions that had considerably augmented the generality of  Theorem \ref{almostcentral}. This work was supported in part by CNCSIS grant nr. 24/28.09.07 (Groups, quantum groups, corings, and representation theory). 
\end{ack}
%%%%%%%%%%%%%%%%%%%%%%%%%%%%%%%%%%%%%%%%%%%%%%%%%%%%%%%%%%%%%%%%%%%%%%%%%%%%%
\section{Braided bialgebras}
\subsection{} Throughout the paper $k$ will denote a field of characteristic zero. Nevertheless all considerations on this paper hold in arbitrary characteristic, the restriction arising from the hypothesis of Masuoka's result, Proposition \ref{masprop}. Therefore, Theorem \ref{thm1}, Theorem \ref{almostcentral}, and Corollary \ref{almostcentral2} also require fields of characteristic zero.

Unless otherwise noted all objects are $k$-vector spaces, all maps are $k$-linear, and all tensor products are over $k$. In principle the relevant categorical framework for the results in this paper is that of symmetric braided tensor categories which arise naturally in a number of contexts. However, everything could and will be done in a somewhat more general context (technically, but not conceptually) in which the existence and functorial properties of the braiding are required strictly only for the objects involved in the statements. We refer to \cite{baez,take} for all the basic definitions and properties below.

As general notation, for $V$ a vector space we denote by $V^{\otimes n}$ its $n$-fold tensor product. If $\c: V^{\otimes 2} \to V^{\otimes 2}$ is a map then $\c_i: V^{\otimes n}\to V^{\otimes n}$ denotes the map $\id_{V^{\otimes i-1}}\otimes \c \otimes \id_{V^{\otimes n-i-1}}$. We will not specify $n$ if it is unambiguous from the context.

\subsection{} A braiding of the vector space $V$ is a map $\c: V^{\otimes 2} \to V^{\otimes 2}$ satisfying the braid equation 
\begin{equation}
\c_1\c_2\c_1= \c_2\c_1\c_2
\end{equation} 
 A  braided vector space is a pair $(V,\c)$, with $\c$ a braiding of $V$. The braiding is said to be symmetric if $\c^2=\id_{V^{\otimes 2}}$; equivalently, we say that $(V,\c)$ is symmetrically braided. We say the braiding $\c$ is trivial if $\c(v\otimes w)=w\otimes v$ for all elements $v,w$ of $V$. A morphism of braided vector spaces $(V,\c_V)$ and $(W , \c_W)$ is a $k$-linear map $f: V\to W$ such that $\c_W(f\otimes f) = (f \otimes f )\c_V$.
 
Let $X$ be a linear subspace of $V$. We say that $X$ is compatible with the braiding $\c$ (or that it is categorical) if $\c(X\otimes V)\subseteq V\otimes X$ and $\c(V\otimes X)\subseteq X\otimes V$. Clearly if two subspaces $X,Y$ of $V$ are compatible with the braiding, then $\c(X\otimes Y)\subseteq Y\otimes X$ and $\c(Y\otimes X)\subseteq X\otimes Y$. We say that a collection of subspaces of $V$  is compatible with the braiding (or categorical) if all its elements are compatible with the braiding.

\begin{Rem} Any reference to the braided space structure on $k$ will make reference to the trivial braiding on $k$.
\end{Rem}

\begin{Rem} If $(V, \c)$ is braided vector space, then $(V^{\otimes 2}, \c_2\c_1\c_3\c_2)$ is also a braided vector space. Unless otherwise noted, for a braided vector space $V$, any reference to the braided structure on $V^{\otimes 2}$ will refer to the structure constructed in this fashion.
\end{Rem}
\begin{Rem} If $(V, \c)$ is braided vector space, the maps $\c_{n,m}:V^{\otimes n}\otimes V^{\otimes m}\to V^{\otimes m}\otimes V^{\otimes n}$  $$\c_{n,m}=(\c_m\dots\c_{n+m-1})(\c_{m-1}\dots\c_{n+m-2})\dots(\c_1\dots\c_n)$$ induce a braiding $\oplus_{n,m}\c_{n,m}$ on $\T(V)$, the tensor algebra of $V$.
\end{Rem}

\begin{Def} A braided algebra is a quadruple $(A,\nabla,1, \c)$ where $(A , \c)$ is a braided vector space, $(A,\nabla, 1)$ is an associative unital algebra, and 
\begin{subequations}
\begin{align}\label{alg1}
\c(\nabla \otimes \id_A ) &= (\id_A\otimes \nabla)\c_1\c_2\\ \label{alg2}
\c( \id_A \otimes \nabla) &= (\nabla \otimes \id_A )\c_2\c_1\\
\c(1 \otimes a) &= a \otimes 1, \c(a \otimes 1) = 1 \otimes a \text{ for all } a \in A 
\end{align}
\end{subequations}
A braided algebra is said to be $\c$-commutative if $\nabla\c=\nabla$.

A morphism of braided algebras is a map which is both a morphism of associative unital algebras and a morphism of braided vector spaces. 
\end{Def}

Note that if $A$ is a braided algebra then the multiplication and the unit are  morphisms  of braided vector spaces. It is convenient to denote $\nabla^{op}=\nabla\c$. It is easy to check that $(A,\nabla^{op},1,\c)$ is also a braided algebra. We use $A^{op}$ to refer to this algebra structure. The algebra $A$ is $\c$-commutative if and only if the two structures coincide.

If $X, Y$ are categorical subspaces of $A$ then \eqref{alg1} and \eqref{alg2} assure that $\nabla(X\otimes Y)$ is also a categorical subspace.

\begin{Def} Let  $(A,\nabla,1,\c)$ be a braided algebra. The map
\begin{equation}
[\cdot,\cdot]_A: A^{\otimes 2}\to A, \quad [\cdot, \cdot]_A=\nabla(\id_{A^{\otimes 2}}-\c)
\end{equation}
is called the braided commutator on $A$. We suppress the subscript from the notation if it is unambiguous from the context.
\end{Def}

It is straightforward to verify that $A$ is $\c$-commutative if and only if the commutator is zero map.

\begin{Def} A braided coalgebra is a quadruple $(C,\Delta,\e, \c)$ where $(C,\c)$ is a braided vector space, $(C ,\Delta,\e)$ is a coassociative counital coalgebra, and 
\begin{subequations}
\begin{align}
(\Delta\otimes \id_C )\c &= \c_2\c_1(\id_C \otimes\Delta)\\ 
(\id_C \otimes\Delta)\c &= \c_1\c_2(\Delta\otimes \id_C)\\ 
(\e \otimes \id_C)\c(c \otimes d) &= \e(d)c = (\id_C \otimes \e)c(d \otimes c ) \text{ for all } c , d \in C 
\end{align}
\end{subequations}
A braided coalgebra is said to be $\c$-cocommutative if $\Delta=\c\Delta$.

A morphism of braided coalgebras is a map which is a morphism of coassociative counital coalgebras and a morphism of braided vector spaces.
\end{Def}

Note that if $C$ is a braided coalgebra then the comultiplication and the counit are morphisms of braided vector spaces. As a consequence $\ker\e$ is a categorical subspace of $C$.

\begin{Rem} If $(A,\nabla,1, \c)$ is braided algebra, then $(A^{\otimes 2}, (\nabla\otimes\nabla)\c_2,1\otimes 1, \c_2\c_1\c_3\c_2)$ is also a braided algebra. Similarly, if $(C,\Delta,\e, \c)$ where $(C,\c)$ is a braided coalgebra, then $(C^{\otimes 2}, \c_2(\Delta\otimes\Delta),\e\otimes \e, \c_2\c_1\c_3\c_2)$. Any reference to the braided algebra structure on $A^{\otimes 2}$ or to the braided coalgebra structure on $C^{\otimes 2}$ will refer to these structures.
\end{Rem}

\begin{Def} A braided bialgebra is a 6-tuple $(H,\nabla, 1, \Delta,\e,\c)$ where $(H ,\nabla,1, \c)$ is a braided algebra, $(H,\Delta, \e,\c)$ is a braided coalgebra, and $\Delta$ and $\e$
are algebra morphisms. A braided Hopf algebra is a braided bialgebra which has a convolution inverse $S: H\to H$ to the identity map.
\end{Def}

It is important to note that for a braided Hopf algebra $H$ the following identities are necessarily true
\begin{subequations}
\begin{align}
(S\otimes \id_H )\c &= \c(\id_H \otimes S),& &(\id_H \otimes S)\c=\c(S\otimes \id_H )\\ 
\nabla\c(S\otimes S) &=S\nabla,& &(S\otimes S)\c\Delta=\Delta S
\end{align}
\end{subequations}

The augmentation ideal $\ker\e$ of a braided bialgebra $H$ is a categorical subspace and so are all its powers. In particular, for any braided bialgebra $H$ the vector space \begin{equation}\Q(H):=\ker\e/(\ker\e)^2\end{equation} inherits a canonical braided vector space structure.

\begin{Lm}\label{lemma1} Let $H$ be a braided bialgebra. Then
\begin{equation}
\Delta[\cdot,\cdot]_H=[\cdot,\cdot]_{H^{\otimes 2}}(\Delta\otimes\Delta)
\end{equation}
\end{Lm}
\begin{proof} Since $\Delta$ is a braided vector space morphism we have \begin{equation}(\Delta\otimes\Delta)\c=\c_2\c_1\c_3\c_2(\Delta\otimes\Delta)\end{equation}
Also, $\Delta$ is a braided algebra morphism hence
\begin{equation}
\Delta\nabla=(\nabla\otimes\nabla)\c_2(\Delta\otimes\Delta)
\end{equation}
Taking both equalities into account we obtain that 
\begin{equation}
\Delta\nabla(\id_{H^{\otimes 2}}-\c)=(\nabla\otimes\nabla)\c_2(\id_{H^{\otimes 4}}-\c_2\c_1\c_3\c_2)(\Delta\otimes\Delta)
\end{equation}
which is precisely our claim.
\end{proof}

\subsection{} An algebra filtration, respectively grading, is said to be a braided algebra filtration, respectively grading, if it is compatible with the braiding. Similarly for coalgebras, bialgebras. 
For $C$ a coalgebra we denote by $\{C_n\}_{n\geq 0}$ the coradical filtration. It is convenient to denote $C_{-1}=0$. Let 
\begin{equation}
\gr C=\oplus_{n\geq 0}(\gr C)(n)=\oplus_{n\geq 0} C_n/C_{n-1}
\end{equation}
be the associated graded coalgebra. If $C$ is a braided coalgebra with categorical coradical the coradical filtration is compatible with the braiding and $\gr C$ thus inherits a braided graded coalgebra structure. Furthermore, the coradical filtration of $\gr C$ is the filtration given by degree and $\gr(\gr C)=\gr C$. 

We say that the (braided) bialgebra $H$ is coradically graded if there exists a (categorical) bialgebra grading of $H$ such that the associated filtration is precisely the coradical filtration. If $H$ is a braided bialgebra and the coradical is a sub-bialgebra then the coradical filtration is a bialgebra filtration and $\gr H$ is a braided coradically graded bialgebra. The braided bialgebra $H$ is said to be irreducible if $H_0=k$.

\subsection{} Let $(V,\c)$ be a braided vector space. The tensor algebra $\T (V)$ with the standard grading and the braiding induced by $\c$ is braided graded algebra.  The quotient of $\T(V)$ by the ideal generated by $(\id_{V^{\otimes 2}}-\c)(V^{\otimes 2})$ is a braided graded $\c$-commutative algebra,  denoted by $\Se(V)$, which will be referred to as the braided symmetric algebra of $V$.

\subsection{} We now come to our main observation.
\begin{Lm}\label{lemma2} Let $H$ be a braided algebra. Then, 
\begin{equation}
[\cdot, \cdot]_{H^{\otimes 2}}=([\cdot,\cdot]_H\otimes\nabla+ \nabla^{op}\otimes[\cdot,\cdot]_H)\c_2+
(\nabla\otimes\nabla)(\id_{H^{\otimes 4}}-\c_2^2)\c_1\c_3\c_2
\end{equation}
\end{Lm}
\begin{proof}
Straightforward computation.
\end{proof}
\begin{Prop}
Let $H$ be a symmetrically braided bialgebra with categorical bialgebra coradical. Assume that \begin{equation}[H_0,H_n]\subseteq H_{n-1}\end{equation} for all $n\geq 0$. Then,
\begin{equation}[H_m,H_n]\subseteq H_{m+n-1}\end{equation}
for all $n,m\geq 0$.
\end{Prop}
\begin{proof} We show that for all $m\geq 0$
\begin{equation}\label{eq1}
[H_m,H_n]\subseteq H_{m+n-1}, \text{ for all } n\geq 0
\end{equation}
by induction on $m\geq 0$. For $m=0$ this is exactly our hypothesis. For $m>0$ we assume that our claim holds for all strictly smaller values and we prove \eqref{eq1} by induction on $n\geq 0$. Again, if $n=0$ the claim follows from the hypothesis and the fact that the commutator is $\c$-skew-symmetric if the braiding $\c$ is symmetric. Let now $n>0$ and assume that the claim holds for all strictly smaller values. The claim is equivalent to 
\begin{equation}\label{eq2}
\Delta([H_m,H_n])\subseteq H_0\otimes H+H\otimes H_{m+n-2}
\end{equation}
From Lemma \ref{lemma1}, Lemma \ref{lemma2} and the fact that the coradical filtration is a categorical we obtain 
\begin{equation}
\Delta([H_m,H_n])\subseteq \sum_{\stackrel{i+j=m}{\scriptscriptstyle s+t=n}} \left([H_i,H_s]\otimes H_{j+t}+H_{s+i}\otimes [H_j,H_t]\right)
\end{equation}
If $s+i\geq 2$ then both terms of the right hand side are subsets of $H\otimes H_{m+n-2}$. If $s+i=0$ then both terms are subsets of $H_0\otimes H$. Finally, if $s+i=1$ then, by the induction hypothesis, the first term is a subset of $H_0\otimes H$ and the second term is a subset of $H\otimes H_{m+n-2}$. In conclusion \eqref{eq2} is satisfied and our claim is proved.
\end{proof}
\begin{Cor} \label{cor}
Let $H$ be a symmetrically braided bialgebra with categorical bialgebra coradical satisfying $$[H_0,H_n]\subseteq H_{n-1}$$ for all $n\geq 0$. Then, the associated braided graded bialgebra $\gr H$ is $\c$-commutative. In particular, if $H$ is an irreducible symmetrically braided bialgebra then $\gr H$ is $\c$-commutative.
\end{Cor}

\subsection{}  Let us recall now two fundamental structural results valid in characteristic zero for symmetric braidings that are due to Kharchenko, and respectively Masuoka. The first result \cite[Theorem 6.1]{khar} is an equivalence between the category of irreducible, $\c$-cocommutative, braided Hopf algebras and the category of $\c$-Lie algebras. This extends the classical equivalence due to Kostant. The second result \cite[Theorem 6.7]{masuoka}, in some sense dual to the first one, is an equivalence between the category of irreducible, $\c$-commutative, braided Hopf algebras and the category of locally nilpotent $\c$-Lie coalgebras. Classically this corresponds to the equivalence between categories of unipotent affine algebraic groups and the category of finite dimensional nilpotent Lie algebras. We will only need here the following partial result  \cite[Proposition 6.8]{masuoka}.
\begin{Prop}\label{masprop}
Let $H$ be an irreducible, $\c$-commutative, symmetrically braided bialgebra. Then, $H$ and $\Se(\Q(H))$ are canonically isomorphic as braided algebras. 
\end{Prop}
Note that if $H$ is a braided graded bialgebra then $\Q(H)$ has a braided vector space structure and this induces a second grading on the braided symmetric algebra; the above isomorphism is a graded morphism for this second graded structure.

We are now ready to prove our main result. See \cite[Theorem 2]{lebruyn} for the trivially braided version.
\begin{Thm}\label{thm1} Let $H$ be an irreducible symmetrically braided bialgebra. Then $H$ is of PBW type as a $k$-module.
\end{Thm}
\begin{proof}
From Corollary \ref{cor} the irreducible symmetrically braided graded bialgebra $\gr H$ is $\c$-commutative and by Proposition \ref{masprop} there exists a canonical braided graded isomorphism between $\gr H$ and $\Se(\Q(\gr H))$. Note that in our situation, the augmentation ideal of $\gr H$ is $\gr_+ H:=\oplus_{n\geq 1} H_n/H_{n-1}$.
\end{proof}
\begin{Rem}
The fact that an irreducible symmetrically braided bialgebra $H$ is of PBW type as a $k$-module does not generally say anything about $H$ having a linear PBW type basis. Of course, if $\gr H$ has a linear PBW type basis generated by homogeneous elements then $H$ has a linear PBW type basis generated by any lifting of the generating set for the basis of $\gr H$. However, as pointed out by several authors (see e.g. \cite[\S 7]{khar}), symmetric braided algebras (even for symmetric braidings) do not generally admit linear PBW type bases (even in a weakened sense). There are though important classes of examples (see below) of symmetric braided algebras that have certain types of PBW bases.
\end{Rem}
\begin{Exp}\label{exp1} Let $G$ be a finite abelian group and let $\chi:G\times G\to k^\times$ be a skew-symmetric bicharacter (i.e. $\chi(g,\cdot)$ and $\chi(\cdot,g)$ are group morphisms and $\chi(g,h)\chi(h,g)=1$ for all $g,h\in G$). On any $G$-graded vector space $V=\oplus_{h\in G}V_h$ we can put a $G$-module structure by letting $g$ act on $V_h$ by multiplication with the scalar $\chi(g,h)$. With this $kG$-module structure and the $kG$-comodule structure induced by the grading $V$ becomes an object in the category ${}_{kG}^{kG}\mathcal{YD}$ of Yetter-Drinfel'd modules over the group algebra $kG$. The category ${}_{kG}^{kG}\mathcal{YD}$ is a braided category and one can verify that on $V\otimes V$ the braiding 
\begin{equation}
\c_{V}:V\otimes V\to V\otimes V
\end{equation}
acts as follows: if $x\in V_g$, $y\in V_h$, then $\c_{V}(x\otimes y)=\chi(g,h)y\otimes x$. The skew-symmetry condition assures that $\c_{V,V}$ is symmetric. The braided symmetric algebra $\Se(V)$ has then a PBW basis in the sense of Scheunert \cite{SchGen}: given a set $\{x_i\}_{i\in I}$ of homogeneous elements of $V$ indexed by a well-ordered set $(I,\leq)$, denote by $g_i$ the degree of $x_i$; the set of elements of the form $x_{i_1}\cdots x_{i_n}$, where $n\geq 0$, ${i_s}\leq {i_{s+1}}$ for all $1\leq s<n$, and ${i_s}<{i_{s+1}}$ if $\chi(g_{i_s},g_{i_s})\neq 1$, is a basis of $\Se(V)$. In consequence, any bialgebra lifting $H$ of $\Se(V)$ (such as the enveloping algebra of a color Lie algebra) has also a PBW basis in the sense of Scheunert. In particular, any irreducible color Hopf algebra has a PBW basis in the sense of Scheunert.
\end{Exp}
\begin{Exp}  The category of super vector spaces has symmetric braiding. Therefore, according to Theorem \ref{thm1} any irreducible super-bialgebra is of PBW type, and moreover, according to Example \ref{exp1}, it also has a PBW basis.
\end{Exp}
%%%%%%%%%%%%%%%%%%%%%%%%%%%%%%%%%%%%%%%%%%%%%%%%%%%%%%%%%%%%%%%%%%%%%%%%%%%%%
\section{Applications}\label{appl} We discuss  a short application of the main result to braided Hopf algebras. All the facts recalled below are well known, especially for trivially braided Hopf algebras; we refer to \cite[\S 5]{masuoka} for the arguments in the general braided case.
\subsection{} Throughout this section $H$ is assumed to be a braided Hopf algebra with  coradical $H_0$; also, $K$ is assumed to be a categorical braided Hopf subalgebra of $H$ that contains the coradical $H_0$. We will use Sweedler's  notation for the coalgebra  structure on $H$. As usual, $\wedge^{n} K$, $n\geq 0$ is defined inductively by $\wedge^0 K:=0$, $\wedge^1 K:=K$, and $\wedge^{n}K:=\Delta_H^{-1}(K\otimes H+H\otimes \wedge^{n-1}K)$ for $n>1$. The filtration $F_nH:=\wedge^{n+1}K$ is a braided Hopf algebra filtration of $H$. We denote by \begin{equation}\gr_K(H):=\oplus_{n\geq 0}(\gr_K(H))(n)=\oplus_{n\geq 0}F_nH/F_{n-1}H\end{equation} the associated graded braided Hopf algebra. We denote by $\nabla,1,\Delta,\e,\c,S$ its relevant operations and structures.

Both the injection $i:K\hookrightarrow \gr_K(H)$ and the projection $\pi: \gr_K (H)\twoheadrightarrow K$ are braided Hopf algebra morphisms.  Let $R$ be the braided graded subalgebra of $\pi$-coinvariants
\begin{equation}\label{eq: Rdef}
R:=\{a_1S(\pi(a_2))~|~a\in \gr_K(H)\}=\{b\in \gr_K(H)~|~\sum b_1\otimes \pi(b_2)=b\otimes 1\}
\end{equation}
With the notation in the Introduction, $R=\gr J$. The linear map $\Pi:\gr_K(H)\to R$, $\Pi(a):=a_1S(\pi(a_2))$ is surjective and its kernel is the left ideal of $\gr_K(H)$ generated by $\ker\e_{|K}$. The braided algebra $R$ acquires several other structures. First, it can be observed that $\Delta(R)\subseteq \gr_K(H)\otimes R$ and hence $R$ is a braided quotient coalgebra of $\gr_K(H)$ with comultiplication given by $\Delta_R=(\Pi\otimes \id_R)\Delta$. It is also a left $K$-module for the action 
\begin{equation}
\ad_K:=\nabla(\nabla\otimes S)(\id_K\otimes\c)(\Delta\otimes\id_R):K\otimes R\to R
\end{equation} 
and a left $K$-comodule for the coaction 
\begin{equation}
\delta_R:=(\pi\otimes\id_R)\Delta:R\to K\otimes R
\end{equation} 
Furthermore, $R$ with the algebra and coalgebra structures mentioned above is an irreducible braided graded Hopf algebra in an appropriately  defined category ${}_{K}^{K}\mathcal{YD}$ of braided Yetter-Drinfel'd modules over $K$. We will only mention here the braiding that makes $R$ a braided Hopf algebra
\begin{equation}\label{braidformula}
\c_R:=(\ad_K\otimes\id_R)(\id_K\otimes\c)(\delta_R\otimes\id_R): R\otimes R\to R\otimes R
\end{equation}
Not only $R$ is irreducible, but also the filtration associated to the grading induced from $\gr_K(H)$ is the coradical filtration of $R$ \cite[Proposition 5.1]{masuoka}; in particular, $R$ is coradically graded and thus $R=\gr R$ and the augmentation ideal of $R$ is $R_+$, its maximal graded ideal. Any set of homogeneous representatives in $R_+$ for a basis of $\Q(R)$ generates $R$ as an algebra. The vector space $\Q(R)=R_+/R_+^2$ is an object of ${}_{K}^{K}\mathcal{YD}$ too. The map
\begin{equation}\label{bosonization}
K\otimes R\to \gr_K(H),\quad k\otimes r\mapsto kr
\end{equation}
is a graded linear isomorphism.
%%%%%%%%%%%%%%%%%%%%%%%%%%%%%%%%%%%%%%%%%%%%%%%%%%%%%%%%%%%%%%%%%%%%%
\subsection{} Let $(A,\Delta_A,\e_A,\c_A)$ and $(B,\nabla_B,1_B,\c_B)$ be a braided coalgebra and a braided algebra, respectively.
\begin{Def} With the notation above, let $f:A\to B$ be a linear map. We say that $f$ is central if 
\begin{equation}
 \nabla_B(f\otimes\id_B)=\nabla_B\c_B(f\otimes\id_B)\quad \text{and}\quad \nabla_B(\id_B\otimes f)=\nabla_B\c_B(\id_B\otimes f)
 \end{equation}
 We say that $f$ is cocentral if 
 \begin{equation}
(f\otimes\id_A) \Delta_A=(f\otimes\id_A)\c_A\Delta_A\quad \text{and}\quad(\id_A\otimes f)\Delta_A=(\id_A\otimes f)\c_A\Delta_A
 \end{equation}
\end{Def}
\begin{Lm}\label{lemma: braidingeq} With the notation above, the following hold
\begin{enumerate}
\item If $i: K\hookrightarrow \gr_K(H)$ is central then $\c_R=\c$.
\item If $\pi: \gr_K(H)\twoheadrightarrow K$ is cocentral then $\c_R=\c$.
\end{enumerate}
\end{Lm}
\begin{proof} The verification of the first statement is somewhat lengthy but nevertheless straightforward; it can be performed using the standard diagrammatic notation. For the second statement keep in mind that since $\gr_K(H)$ is a graded braided Hopf algebra we have 
\begin{equation}(\pi\otimes \id_{\gr_K(H)})\c=\c(\id_{\gr_K(H)}\otimes i\pi)
\end{equation}
and using the hypothesis and \eqref{eq: Rdef} we obtain that $\delta_R=\c(\id_R\otimes i\pi)\Delta=1\otimes \id_R$ from which our claim follows.
\end{proof}
\begin{Thm}\label{almostcentral}
With the notation above, assume that the braiding $\c$ is symmetric. If either $i: K\hookrightarrow \gr_K(H)$ is central, or $\pi: \gr_K(H)\twoheadrightarrow K$ is cocentral then $H$ is of PBW type as a $K$-module.
\end{Thm}
\begin{proof}
From Lemma \ref{lemma: braidingeq} and Theorem \ref{thm1} we obtain that $R$ is isomorphic to the braided symmetric algebra $\Se(\Q(R))$.
\end{proof}

Under the same hypotheses, for usual (i.e. trivially braided Hopf algebras) $Q(R)$ is also trivially braided and we obtain that $H$ has a PBW basis as a $K$-module. In this situation we obtain the following extension of the main result in \cite{ion,ion2}.

\begin{Cor}\label{almostcentral2}
Under the hypotheses of Theorem \ref{almostcentral}, assume in addition that $H$ is trivially braided. Then, $H$ has a PBW basis as a $K$-module.
\end{Cor}
%%%%%%%%%%%%%%%%%%%%%%%%%%%%%%%%%%%%%%%%%%%%%%%%%%%%%%%%%%%%%%%%%%%%%%%%%%%%%

\end{document}